\pdfoutput=1
\documentclass[reqno]{amsart}

\usepackage[utf8]{inputenc}
\usepackage{mathrsfs, mathtools}
\usepackage{amsfonts, amssymb, amsmath, amsthm}
\usepackage{microtype}
\usepackage{hyperref}
\hypersetup{final, colorlinks=true, citecolor=blue}
\usepackage{url}
\usepackage{bbm}
\usepackage{enumitem}
\usepackage{xifthen}
\usepackage{tensor}
\usepackage[normalem]{ulem}

\newcommand{\Q}{\mathbb{Q}}
\newcommand{\R}{\mathbb{R}}
\newcommand{\C}{\mathbb{C}}
\newcommand{\Z}{\mathbb{Z}}
\newcommand{\HH}{\mathcal{H}}
\newcommand{\cl}{\operatorname{Cl}}

\newcommand{\Tr}{\operatorname{Tr}}
\newcommand{\GL}{\operatorname{GL}}
\newcommand{\SL}{\operatorname{SL}}
\newcommand{\Bil}{\operatorname{Bil}}

\newcommand{\Stab}{\operatorname{Stab}}

\def\ll#1{{\left\langle{#1}\right\rangle}}
\def\id#1{{\mathfrak{#1}}}
\def\idn#1{{\norm(\mathfrak{#1})}}
\def\hatx#1{{{\widehat{#1}}^\times}}
\newcommand{\p}{_{\id p}}
\newcommand{\pv}{{\id p}}
\newcommand{\bbu}{\mathbbm{1}}
\newcommand{\ssl}{^{l}}
\newcommand{\sslu}{^{l,u}}
\newcommand{\tlu}{\vartheta\sslu}
\newcommand{\etalD}{\eta^l_D}
\newcommand{\psilD}{\psi^l_D}

\newcommand{\mba}{\mathbf{a}}
\newcommand{\mbf}{\mathbf{f}}
\newcommand{\mbk}{\mathbf{k}}
\newcommand{\pip}{\pi_{\id p}}

\def\smat#1{\left(\begin{smallmatrix} #1 \end{smallmatrix} \right)}
\def\pmat#1{\left(\begin{matrix} #1 \end{matrix} \right)}
\def\abs#1{\left\vert{#1}\right\vert}
\def\num#1{\left\vert{#1}\right\vert}
\def\numset#1{\num{\left\{{#1}\right\}}}

\def\omegaN{\omega(\id N)}

\newcommand{\kro}[2]{\left(\frac{#1}{#2}\right)}
\newcommand{\subp}[1][]{\ifthenelse{\equal{#1}{}}{_{\id p}}{_{\id{#1}}}}
\newcommand{\Mhk}[1][]{\mathcal{M}^{#1}_{\mathbf{3/2+k}}(4\id{N},\chic{-u})}
\newcommand{\Shk}[1][]{\mathcal{S}^{#1}_{\mathbf{3/2+k}}(4\id{N},\chic{-u})}
\newcommand{\Sk}{\mathcal{S}_{\mathbf{2} + 2\mathbf{k}}(\id{N})}

\DeclareMathOperator{\vardot}{\,\cdot\,}
\DeclareMathOperator{\norm}{\mathcal{N}}
\DeclareMathOperator{\trace}{\mathcal{T}}
\DeclareMathOperator{\disc}{\Delta}
\newcommand{\sgn}{\operatorname{sgn}}
\newcommand{\chic}[2][]{\chi_{#1}^{#2}}
\newcommand{\cond}[1]{\mathfrak{f}_{#1}}
\newcommand{\OO}[1][]{\mathcal{O}_{\mkern-2mu #1}}
\newcommand{\OOx}[1][]{\mathcal{O}_{\mkern-3mu #1}^{\times}}
\newcommand{\What}{\widehat{W}}
\newcommand{\hatOO}[1][]{\widehat{\mathcal{O}}_{\mkern-2mu #1}}
\newcommand{\hatOOx}[1][]{\hatx{\mathcal{O}}_{\mkern-3mu #1}}
\def\Fnc{F^\times \backslash (F^\times)^2}
\def\Fc{(F^\times)^2}

\def\dF{d_{\mkern-2mu F}}
\def\IF{\mathcal{J}_{\mkern-2mu F}}
\def\hF{h_{\mkern-2mu F}}
\def\tK{t_{\mkern-2mu K}}
\def\mK{m_{\mkern-2mu K}}
\def\muR{\mu_{\mkern-2mu R}}
\def\Tm{\mathop{T_{\mkern-1mu \id{m}}}}
\def\Tg{\mathop{T_{\mkern-2mu g}}}

\def\sC{\mathcal C}
\def\DC{\mathcal D(\sC)}

\theoremstyle{plain}
\newtheorem{prop}{Proposition}
\newtheorem{thm}{Theorem}
\newtheorem{coro}{Corollary}
\newtheorem{lemma}{Lemma}

\theoremstyle{remark}
\newtheorem{rmk}{Remark}

\numberwithin{prop}{section}
\numberwithin{equation}{section}
\numberwithin{lemma}{section}
\numberwithin{coro}{section}
\numberwithin{thm}{section}
\numberwithin{rmk}{section}

\begin{document}

\title
[An explicit Waldspurger formula for Hilbert modular forms II]
{An explicit Waldspurger formula\\ for Hilbert modular forms II}

\author{Nicol\'as Sirolli}
\address{Universidad de Buenos Aires and IMAS-CONICET, Buenos Aires, Argentina}
\email{nsirolli@dm.uba.ar}

\author{Gonzalo Tornar\'ia}
\address{Universidad de la República, Montevideo, Uruguay}
\email{tornaria@cmat.edu.uy}

\keywords{Waldspurger formula, Hilbert modular forms, Shimura
cor\-re\-spon\-dence}
\subjclass[2010]{Primary: 11F67, 11F41, 11F37}

\begin{abstract} 

We describe a construction of preimages for the Shimura map on
Hilbert modular forms using generalized theta series, and give an explicit
Waldspurger type formula relating their Fourier coefficients to central values
of twisted $L$-functions.
This formula extends our previous work, allowing to compute these central values
when the main central value vanishes.
\end{abstract}

\maketitle

\section{Introduction}

Let $g$ be a normalized Hilbert cuspidal newform over a totally real number 
field $F$, of non-trivial level $\id N$, weight $\mathbf{2}$ + 
2$\mathbf{k}$ and trivial central character.
For each $\id p \mid \id N$ denote by $\varepsilon_g(\id p)$ the eigenvalue of
the $\id p$-th Atkin--Lehner involution acting on $g$, and let $\mathcal{W}^{-} =
\{\id p \mid \id N\,:\,\varepsilon_g(\id p) = -1\}$. We make the following
hypothesis on $g$.

\begin{enumerate}[label=\textbf{Hg.},ref={Hg},
                  labelindent=\parindent,leftmargin=*,
                  topsep=\medskipamount
                  ]
   \setcounter{enumi}{1}
    \item  $v\p(\id N)$ is odd for every $\id p \in \mathcal{W}^-$.\label{hyp:JL}
\end{enumerate}

For $l,D \in F^\times$ such that $\Delta =lD$ is totally negative, denote by
$L_{l,D}(s,g) = L(s,g\otimes \chic{l}) L(s,g\otimes \chic{D})$ 
the Rankin--Selberg convolution \textit{L}-function of $g$ by the genus character
associated to the pair $(l,D)$, normalized with center of symmetry at $s=1/2$.
The main result of this article is stated in Theorem~\ref{thm:main_thm}; in the
simpler form given by Corollary~\ref{coro:main} it claims that,
under certain hypotheses on $l$,
there exists a Hilbert cuspidal form $f\sslu$ in the Kohnen plus subspace of weight
$\mathbf{3/2+k}$ and level $4 \id N$ whose Fourier coefficients
$\lambda(-uD,\id a;f\sslu)$ satisfy
 \[
	 L_{l,D}(1/2,g)= \frac{c_{\Delta}}{(-\Delta)^\mathbf{k+1/2}} \,
	 \abs{ \lambda(-uD,\id a;f\sslu) }^2
 \]
for every \emph{permitted} $D$ such that $\chi^l$ and $\chi^D$ have coprime
conductors.
Here $u$ is an auxiliary unit such that $lu$ is totally positive;
$c_{\Delta}$ and $\id a$ are, respectively, an 
explicit positive rational number and a fractional ideal of $F$ depending only,
respectively, on $\Delta$ and on $D$.
Furthermore if $L(1/2,g\otimes\chi\ssl) \neq 0$, which can be achieved by
choosing $l$ suitably, then $f\sslu \neq 0$ and, at least when $\id N$ is odd
and square-free, it maps to $g$ under the Shimura correspondence.
Actually, in this case we construct a linearly independent family of preimages
for this correspondence, as shown in Corollary \ref{coro:family}.

In our previous article \cite{waldspu} we consider the case $l=1$, in which the
above formula gives no information about the values $L(1/2,g\otimes \chic{D})$
when $L(1/2,g) = 0$. In this article we get rid of this restriction, by using
generalized theta series which give the cusp form $f\sslu$. These theta series
are defined using a generalization of the weight functions introduced
in \cite{mrvt,tornatesis}.

Furthermore, the formula given in \cite{waldspu} required that 
$\num{ \mathcal{W}^- }$ and $[F:\Q]$ had the same parity,
and that $(-1)^{\mathbf{k}} = 1$.
These hypotheses are no longer required, as long as $l$ is chosen suitably (see
Hypotheses~\ref{hyp:paridad} and \ref{hyp:peso_signo} in
Sect.~\ref{sect:main}).

Besides extending \cite{waldspu}, and thus the explicit formulas with $l=1$
given in \cite{gross,bsp} when $F = \Q$ and in \cite{Xue-Waldspu} for more
general $F$, our main result implies part of \cite[Theorem 1.1]{mao}, where the
author considers the case $F = \Q$, odd prime level and weight $2$.
In that setting, the case $D>0$ is covered by Mao, but not (yet) by our formula.

The proof of our main result is based on the formulas of \cite{zhang-gl2} and
\cite{Xue-Rankin} giving central values of Rankin $L$-functions in terms of
geometric pairings; we relate the latter to Fourier coefficients of
half-integral weight modular forms by considering special points, as it is done
in the classical setting (see \cite{gross}).

\medskip
This article is organized as follows.
In Sect.~\ref{sect:quat_forms} we recall some basic facts about the space of
quaternionic modular forms.
In Sect.~\ref{sect:half_int} we introduce weight functions and we show how to
obtain half-integral weight Hilbert modular forms out of these functions and
quaternionic modular forms.
In Sect.~\ref{sect:special_pts} we give a formula for the Fourier
coefficients of these half-integral weight modular forms in terms of special
points and the height pairing.
In Sect.~\ref{sect:height_geom} we relate central values of twisted
$L$-functions to the height pairing.
In Sect.~\ref{sect:orders} we state an auxiliary result, needed for the proof
of the main result of this article, which we give in Sect.~\ref{sect:main}.
Finally, in the Epilogue we show that when an auxiliary parameter as $u$ above
does not exist, we can still compute central values in terms of Fourier
coefficients of certain skew-holomorphic theta series.

We used our main result to compute the central values $L(1/2,g\otimes\chic{D})$
in some examples where $L(1/2,g)$ vanishes. 
Our calculations were done with Sagemath~\cite{sage} and Magma~\cite{magma}.
The data obtained can be found in \cite{code}.

\subsection*{Notation summary}

We fix a totally real number field $F$ of discriminant $\dF$,
with ring of integers $\OO$ and different ideal $\id{d}$.
We denote by $\IF$ the group of fractional ideals of $F$, and we write $\cl(F)$
for the class group, $C_F$ for the idele class group, and $\hF$ for the class 
number.
We denote by $\mathbf{a}$ the set of embeddings $\tau : F \hookrightarrow \R$,
and by $\mathbf{f}$ the set of nonzero prime ideals $\pv$ of $F$.
For $\xi \in F^\times$ we let 
$F^\xi = \{\zeta \in F^\times\,:\,\sgn(\xi_\tau \zeta_\tau) = 1 \; \forall\,
\tau \in \mathbf{a}\}$, and
we let $\sgn(\xi) = \prod_{\tau \in \mba} \sgn(\xi_\tau)$.
Given $\mathbf{k} = (k_\tau) \in \Z^\mathbf{a}$ and $\xi \in F$, 
we let $\xi^\mathbf{k} = \prod_{\tau \in \mathbf{a}} \xi_\tau^{k_\tau}$.
We use $\pv$
as a subindex to denote completions of global objects at $\pv$, as well as to
denote local objects.
Given an integral ideal $\id N \subseteq \OO$ we let
$\omegaN = \numset{\id p\in\mathbf{f}\,:\,\id p \mid \id N}$.
Given $\id p$ we denote by $\pip$ a local uniformizer at $\id p$, and we let 
$v\p$ denote the $\id p$-adic valuation.

Given a character $\chi$ of $C_F$, we denote
\[
 \chi_\mathbf{a} = \prod_{\tau \in \mathbf{a}} \chi_\tau, \qquad
 \chi_\mathbf{c} = \prod_{\pv \mid \id{c}} \chi\p, 
\]
for every $\id c \subseteq \OO$. If $\chi$ has conductor $\id f$, we let 
$\chi_*$ denote the induced character on ideals prime to $\id f$. Given $\xi \in 
F^\times$, we denote by the $\chic{\xi}$ the character of 
$C_F$ corresponding to the extension $F(\sqrt{\xi})/F$, and we denote 
its conductor by $\cond{\xi}$. 

Given a quadratic extension $K/F$, we let $\OO[K]$ be the maximal order.
If $K/F$ is totally imaginary we let $\tK = [\OOx[K]:\OOx]$, and let $\mK \in
\{1,2\}$ be the order of the kernel of the natural map $\cl(F)\to\cl(K)$.
For $x\in K$ we let $\disc(x)=(x-\overline{x})^2$.
If $K = F(\sqrt \xi)$ with $\xi \in \Fnc$, given $\id a \in \IF$, we say that
the pair $(\xi,\id a)$ is a \emph{discriminant} if there exists $\omega \in K$
with $\disc(\omega) = \xi$ such that $\OO \oplus \id a\,\omega$ is an order in $K$.
When this order equals $\OO[K]$,
we say that the discriminant $(\xi,\id a)$ is \emph{fundamental}. 
For completeness, we also say that $(\xi^2,\xi^{-1} \OO)$ is a fundamental
discriminant.

Given a quaternion algebra $B/F$ we denote by $\norm:B^\times\to F^\times$ and
$\trace:B\to F$ the reduced norm and trace maps, and we use $\norm$ and $\trace$
to denote other norms and traces as well. 
We denote by $\widehat B = \prod'\subp B\p$ and $\hatx B = \prod'\subp
B\p^\times$ the corresponding restricted products, and we use the same notation
in other contexts.

Finally, given a level $\id N \subseteq \OO[F]$, an integral or half-integral
weight $\mathbf k$ and a character $\chi$ of $C_F$ of conductor dividing $\id
N$, we denote by $\mathcal{M}_{\mathbf{k}}(\id N, \chi)$ and
$\mathcal{S}_{\mathbf{k}}(\id N, \chi)$ the corresponding spaces of Hilbert
modular and cuspidal forms.
When $\mbk$ is half-integral, given $f \in \mathcal{M}_{\mathbf{k}}(\id N,
\chi)$, for $\xi \in F^+ \cup \{0\}$ and $\id a \in \IF$ we denote by
$\lambda(\xi,\id a;f)$ its $\xi$-th Fourier coefficient at $\id a$.

\section{Quaternionic modular forms}\label{sect:quat_forms}

Let $B$ be a totally definite quaternion algebra over $F$. Let $(V, \rho)$ be 
an irreducible unitary right representation of $B^\times/F^\times$, which we 
denote by $(v,\gamma)\mapsto v\cdot \gamma$. Let $R$ be an order
in $B$. A \emph{quaternionic modular form} of weight 
$\rho$ and level $R$ is a function $\varphi: \hatx B\to V$ such that for every 
$x\in \hatx B$ the following transformation formula is satisfied:
\[
 \varphi(u x \gamma) = \varphi(x) \cdot \gamma \qquad
 \forall\,u\in\hatx R,\,\gamma\in B^\times\,.
\]
The space of all such functions is denoted by $\mathcal{M}_\rho(R)$. We let 
$\mathcal{E}_\rho(R)$ be the subspace of functions that factor through the map 
$\norm:\hatx B \to \hatx F$. These spaces come equipped with the action 
of Hecke operators $\Tm$, indexed by integral ideals $\id m \subseteq \OO$,
and given by
\[
\Tm\varphi(x) = \sum_{h \in \hatx R \backslash H_\id{m}} \varphi(h x) 
\,,
\]
where $H_\id{m} = \left\{h \in \widehat R \, : \, \hatOO
\norm(h) \cap \OO = \id m\right\}$.

Given $x \in \hatx B$, we let
\[
 \widehat R_x = x ^{-1} \widehat R \, x , \qquad R_x = B\cap \widehat R_x, 
 \qquad \Gamma_{\!x} = R_x^\times / \OOx,
 \qquad t_x = \num{ \Gamma_{\!x} }\,.
\]
The sets $\Gamma_{\!x}$ are finite since $B$ is totally definite. Let $\cl(R) = 
\hatx R \backslash \hatx B / B^\times$. We define an inner product on 
$\mathcal{M}_\rho(R)$, called the \emph{height pairing}, by
\[
 \ll{\varphi,\psi} = \sum_{x\in \cl(R)} \tfrac{1}{t_x}
 \ll{\varphi(x),\psi(x)}\,. 
\]
The space of \emph{cuspidal forms} $\mathcal{S}_\rho(R)$ is defined as the 
orthogonal complement of $\mathcal{E}_\rho(R)$ with respect to this pairing.

Let $N(\widehat R) = \{z \in \hatx B\,:\,\widehat R_z = \widehat R\}$ be the 
normalizer of $\widehat R$ in $\hatx B$.  We let $\widetilde \Bil(R) = \hatx 
R \backslash N(\widehat R) / F^\times$. We have an embedding $\cl(F) 
\hookrightarrow \widetilde \Bil(R)$. The group $\widetilde \Bil(R)$, and in 
particular $\cl(F)$, acts on $\mathcal{M}_\rho(R)$ by letting $(\varphi\cdot 
z)(x)=\varphi (zx)$. This action restricts to $\mathcal{S}_\rho(R)$ and is 
related to the height pairing by the equality 
\begin{equation}\notag
\ll{\varphi \cdot z,\psi \cdot z} = \ll{\varphi, \psi}\,.
\end{equation}

The subspaces of $\mathcal{M}_\rho(R)$ and $\mathcal{S}_\rho(R)$ fixed by the 
action of $\cl(F)$ are denoted by $\mathcal{M}_\rho(R,\bbu)$ and 
$\mathcal{S}_\rho(R,\bbu)$. Let $\Bil(R) = \hatx R \backslash N(\widehat R) / 
\hatx F$. Then $\Bil(R)$ acts on $\mathcal{M}_\rho(R,\bbu)$ and 
$\mathcal{S}_\rho(R,\bbu)$.
Given a character $\delta$ of $\Bil(R)$ we denote
\[
  \mathcal{M}_\rho(R,\bbu)^\delta =
  \left\{ \varphi \in \mathcal{M}_\rho(R,\bbu) \; : \;
  \varphi \cdot z = \delta(z) \, \varphi \quad \forall\, z \in \Bil(R)\right\}.
\]

\subsection*{Forms with minimal support}
Given $x\in \hatx B$ and $v\in V$, let $\varphi_{x,v} \in\mathcal{M} 
_\rho(R)$ be the quaternionic modular form given by
\[
 \varphi_{x,v}(y) = \sum_{\gamma\in \Gamma_{\!x,y}} v \cdot \gamma\,,
\]
where $\Gamma_{\!x, y} = (B^\times \cap x ^{-1} \hatx R y) / \OOx$. 
Note that $\varphi_{x,v}$ is supported in $\hatx R x B^\times$. Furthermore, we 
have that 
\begin{alignat}{2}
	\varphi_{ux\gamma, v} & = \varphi_{x,v\cdot\gamma^{-1}}
	&& \forall \, u\in\hatx R, \, \gamma \in B^\times,
	\label{eqn:transf_phi2}
	\\
	\varphi_{x,v} \cdot z & = \varphi_{z^{-1} x, v}
	&\qquad& \forall \, z \in \widetilde\Bil(R)\,.
	\label{eqn:transf_phi}
\end{alignat}
Given $\varphi \in \mathcal{M}_\rho(R)$, using that $\varphi(x) \in 
V^{\Gamma_{\!x}}$ for every $x \in \hatx B$ we get that
\begin{equation}\label{eqn:phi=sum_phix}
 \varphi = \sum_{x\in \cl(R)} \tfrac{1}{t_x} \varphi_{x,\varphi(x)}\,.
\end{equation}

The following two results are proved in \cite[Sect.~1]{waldspu}.

\begin{prop}\label{prop:hecke_on_phi}
 
Let $x\in \hatx B$ and $v\in V$. Then
 $\Tm\varphi_{x,v} = \sum_{h\in H_\id{m} / \hatx R}  \varphi_{h^{-1}x,v}$\,.

\end{prop}

\begin{prop}\label{prop:height_on_phi}
Given $x,y\in\hatx B$ and $v,w\in V$, we have
\[
 \ll{\varphi_{x,v}\,, \varphi_{y,w}} 
 = \sum_{\gamma\in\Gamma_{\!x,y}} \ll{v\cdot\gamma, w}.
\]
\end{prop}

\section{Weight functions and half-integral weight modular forms}\label{sect:half_int}

We fix $l \in F^\times$ and $\id N \subseteq \OO$. We assume that $\cond l$ is
prime to $2 \id{d} \id N$.
In particular, $\cond l$ is square-free.
We let $\id b \in \IF$ be such that $(l,\id b) $ is a fundamental discriminant.
Such $\id b$ exists and is unique: see \cite[Proposition 2.11]{waldspu}.
Furthermore, assuming that it exists, we fix an auxiliary $u \in \OO^\times \cap
F\ssl$.
Note that $\cond l$ is prime to $\cond u$, since $\cond u \mid 4$.

Let $W = B/F$, in which we consider the
totally negative definite ternary quadratic form
$\Delta(x)=\trace(x)^2-4\norm(x)$.
Let $L \subseteq W$ be an integral lattice of level $\id N$
(i.e., $\Delta(L)\subseteq \OO$ and $\Delta(L^\sharp) \subseteq \id{N}^{-1}$).
For each prime $\id p \mid \cond l$, we let $w\p$ be a (local) \emph{weight function}
on $L\p$.
This is a function $w\p : L\p / \pi\p L\p \to \{0,1,-1\}$ such that the 
following conditions hold for every $x \in L\p$:
\begin{align}
	w\p(x) & = 0 \quad \text{if } \Delta(x) \notin \pi\p \OO[\id p].
	\label{eqn:wnormal}\\
	w\p (\xi x) & = \chic[\pv]{l}(\xi) \, w\p(x)
	\quad \text{for every }\xi \in \OO[\pv]^\times.
	\label{eqn:whomog}\\
	w\p(y^{-1} x y) & = \chic[\pv]{l}(\norm y) \, w\p(x) \quad \text{for every }y \in
	B\p^\times \text{ such that }y^{-1} L\p y = L\p.
	\label{eqn:wtransp}
\end{align}
In \cite[Sect.~4]{tornatesis} it is shown that nonzero weight functions do
exist, and that they also satisfy:
\begin{align}
	w\p(x) & \neq 0 \quad \textrm{if } x \notin \pi\p L\p \textrm{ and }
	\Delta(x) \in \pi\p \OO[\pv].\label{eqn:wvector}
\end{align}
We let $w : \What \to \{0,1,-1 \}$ be the (adelic) weight function on
$\widehat L$ given by
\begin{equation*}
	w(x) = \begin{cases}
		\prod_{\pv \mid \cond l} w\p(x\p), & \text{if } x \in \widehat L, \\
		0, & \text{otherwise}.
	\end{cases}
\end{equation*}
Note that when $l \in \Fc$ the function $w$ is simply the characteristic
function of $\widehat{L}$. In general, $w$ is supported in $\widehat
L$ and it is $\cond l\widehat L$-periodic.

Finally, given $\id a \in \IF$ we let $w(\vardot;\id a)$
denote the weight function on $\id{a}^{-1}\widehat L$
given by $w(x;\id a) = \chic{l}(\xi)\,w(\xi x)$,
where $\xi \in \hatx F$ is such that $\xi \hatOO \cap F = \id a$.
Note that by \eqref{eqn:whomog} if $\id m \subseteq \OO$ is prime to $\cond l$
we have that
\begin{equation}
	w(\vardot; \id m \id a) \vert_{\id{a}^{-1}\widehat L}
	= \chic[*]{l}(\id m)\,w(\vardot; \id a).
	\label{eqn:wsegda}
\end{equation}

\medskip

Denote by $\HH$ the complex upper half-plane.
We consider the exponential function given by
\[
	e_F : F\times \HH^\mathbf{a} \to \C, \quad (\xi, z) \mapsto 
	\exp{\left(2\pi i \Big(
				\sum_{\tau\in \mba} \xi_\tau\, z_\tau 
	\Big)\right)}.
\]
Given a homogeneous polynomial $P$ in $W$ of
degree $\mathbf{k}$, harmonic with respect to the quadratic form $\Delta$,
we let $\tlu:\HH^\mathbf{a}\to \C$ be the function given by
\begin{equation}\label{eqn:theta}
	\tlu(z) = \tlu(z;w,P)
	= \sum_{y\in \id b^{-1}L} w(y; \id b) \, P(y)
	\, e_F(-u\disc(y) /l,z/2)\,.
\end{equation}
Note that $\tlu$ is trivially zero if $\sgn(l) \neq (-1)^\mbk$, 
since in that case $w(\vardot; \id b) \, P(\vardot)$ is an odd function.
When $l = u = 1$ we recover the theta series considered in \cite{waldspu}.

Furthermore, for $D \in F^{-l} \cup \{0\}$ and $\id a \in \IF$ we let
\begin{equation}\label{eqn:fourier_coeffs}
	c(D, \id a; w,P) = \frac1{\idn a}
	\sum_{y \in \mathcal{A}_{\Delta,\id c}(L)} w(y;\id c) \, P(y) \,,
\end{equation}
where $\Delta = lD$ and $\id c = \id{ab}$.
Here, for $\Delta \in F^- \cup \{0\}$ and $\id c \in \IF$ we denote
\[
 \mathcal{A}_{\Delta,\id c}(L) = \{y \in \id c^{-1}L\,:\, \disc(y) = \Delta\}.
\]
Note that \eqref{eqn:fourier_coeffs} does not depend on $u$.

\begin{prop}\label{prop:modularity}
With the notation as above, 
\begin{enumerate}
\item $\tlu \in \Mhk$.
\item For $D\in F^{-l} \cup\{0\}$ and $\id a \in \IF$, we have that
	\begin{equation}\label{eqn:cfourier}
		\lambda(-uD,\id a; \tlu) = c(D, \id a;w,P).
	\end{equation}
\item$\tlu$ is a cusp form, provided $l\notin \Fc$ or
	$\mbk \neq \mathbf{0}$.
\end{enumerate}

\end{prop}

The remarkable fact about this modularity result, which was proved in the case 
$F = \Q$, $l$ a positive prime and $u = 1$ in \cite{tornatesis}, is that the
level of $\tlu$ does not depend on $l$.
To prove it we will follow closely \cite[Sect.~11]{shim-hh}, to which we
refer for the details we omit. We start by setting some notation.

\medskip

We consider the exponential functions defined by
\begin{align*}
	e\p & : F\p \to \C, & x\p & \mapsto
	\exp\left(-2\pi i (\Tr_{F\p/\Q_p}(x\p))\right), & \pv \mid p, \\
	e_\mbf & : \widehat F \to \C, & x & \mapsto \prod_{\pv \in \mbf} e\p(x\p).
\end{align*}
Note that given $\xi \in F$, then $e_\mbf(\xi \zeta) = 1$ for every $\zeta \in
\OO$ if and only if $\xi \in \id d^{-1}$.

Given $\id b, \id c \in \IF$ such that $\id{bc} \subseteq \OO$, we denote
\begin{align*}
	\Gamma\p [\id b, \id c] & = \left\{ \beta \in \SL_2(F\p) :
		a_\beta,d_\beta \in \OO[\pv], b_\beta \in \id b\p, c_\beta \in \id c\p
    \right\}, \\
  \widehat \Gamma[\id b, \id c] & = \SL_2(\widehat F) \cap \prod_{\pv \in
  \mathbf{f}} \Gamma\p[\id b, \id c], \\
  \Gamma[\id b, \id c] & = \SL_2(F) \cap \widehat \Gamma[\id b, \id c].
\end{align*}
Furthermore, given $\id a \subseteq \OO$  we denote
$\Gamma[\id a] = \Gamma[2 \id{d}^{-1},2^{-1} \id{d} \cdot \id a]$.
For $\beta \in \SL_2(\widehat F)$ denote by $\id{a}_\beta$ the fractional 
ideal which is locally given by $c_{\beta\p} \id{d}\p^{-1} + d_{\beta\p} 
\OO[\pv]$. Note that $\id{a}_\beta = \OO$ whenever $\beta \in 
\widehat \Gamma[\id{d}^{-1},\id{d}]$.

\medskip

Fix a basis of $W$, through which we identify $W$ with $F^3$. 
Let $S \in \GL_3(F)$ be the matrix of the quadratic form $-u\disc/l$ with respect
to this basis.
Since $u \in F^l$ we have that $S$ is totally positive definite. 
Furthermore, $\det(S) = lu \in F^\times / \Fc$. 
These facts are used in Proposition \ref{prop:shimU} and
\eqref{eqn:shim_automorph} below.

There is a left action, depending on $S$, of (a certain subgroup of) $\SL_2(F)$
on the Schwartz-Bruhat space $\mathcal S(\widehat{F}^3)$, which we denote by
$(\beta, \eta) \mapsto {^{\beta}\eta}$. 
Shimura proves the following results regarding this action.

\begin{prop}

Let $\eta \in \mathcal S(\widehat{F}^3)$.

\begin{enumerate}
\item Let $\beta = \left(\begin{smallmatrix} 1 & b \\ 0 & 1 
\end{smallmatrix} \right)$, with $b \in F$. Then
\begin{equation}\label{eqn:weil_parab}
	{^{\beta}{\eta}(x)}= e_{\mathbf f} (x^t S x \, b/2) \,\eta(x).
\end{equation}
\item Let $\iota = \left( \begin{smallmatrix} 0 & -1 \\ 1 & 0\end{smallmatrix} 
\right)$. Then
\begin{equation}\label{eqn:weil_fourier}
	{^{\iota}\eta}(x) = c_S \int_{\widehat{F}^3} \eta(y) 
	e_{\mathbf f} (-x^t S y) \, dy,
\end{equation}
where $c_S$ is a nonzero constant depending on $S$.
\end{enumerate}

\end{prop}

\begin{prop}\label{prop:shimU}
Let $\eta\in\mathcal{S}(\widehat{F}^3)$. There is an open subgroup $U$ of 
$\widehat \Gamma[2\id{d}^{-1},2 \id{d} \cdot \cond{lu}]$  such that if 
$\beta\in \SL_2(F)\cap\left(\begin{smallmatrix} \xi & 0\\ 0 & 
\xi^{-1}\end{smallmatrix}\right) U$ with $\xi\in \hatx F$, then
\[
{^{\beta}\eta}(x) = \chic[\mathbf{a}]{lu} (d_\beta) \chic[*]{lu}(d_\beta 
\id{a}_\beta^{-1}) \norm(\id{a}_\beta)^{3/2} \eta(\xi x) \qquad \forall\,x\in 
\widehat{F}^3.
\]
\end{prop}

\bigskip

The function $\eta \in \mathcal S(\widehat{F}^3)$ we consider from now on is
$\eta(x) = w(x;\id b)$,
where we identify $\widehat{F}^3$ with $\What$.
Note that $\eta$ is supported in $\id b^{-1}\widehat{L}$,
and it is $l\id b\widehat{L}$-periodic.

\begin{lemma}\label{lem:weil_casos}
Let $x \in \widehat{F}^3$.
\begin{enumerate}
\item $\eta(\xi x) = \chic[\cond l]{l}(\xi) \, \eta(x)$ for every $\xi \in \hatOOx$.
\item If $\eta(x) \neq 0$, then $\Delta(x) / l \in \hatOO$.
\item ${^{\iota}\eta}(\xi x) = \chic[\cond l]{l}(\xi) \, {^{\iota}\eta}(x)$
	for every $\xi \in \hatOOx$.
\item If ${^{\iota}\eta}(x) \neq 0$, then $\Delta(x) / l \in 
\id{d}^{-2}\id N ^{-1} \hatOO$.
\end{enumerate}
\end{lemma}

\begin{proof}

The proofs of (a) and (b) follow immediately from \eqref{eqn:whomog} and
\eqref{eqn:wnormal}, respectively.
Furthermore, (c) follows from (a) and \eqref{eqn:weil_fourier}.

To prove (d), we first note that since $\eta$ is $l\id b 
\widehat L$-periodic by \eqref{eqn:weil_fourier} we have that
\[
{^{\iota}\eta}(x) = e_{\mathbf{f}}(x^t Sl y) \, {^{\iota}\eta}(x)
\]
for every $y \in \id b \widehat L$.
Hence if ${^{\iota}\eta}(x) \neq 0$, since $u$ is a unit we have that 
$x^t Sl y \in \id{d}^{-1}$ for every $y \in \id b \widehat L$,
i.e. $x \in (\id d \id b \widehat L)^\sharp = \id d^{-1} \id b^{-1}
\widehat L^\sharp$.
Since $L$ has level $\id N$, this implies that $\Delta(\zeta x) \in 
\id{d}^{-2}\id N ^{-1} \hatOO$,
where we let $\zeta \in \hatx F$ be such that $\zeta\hatOO \cap F = \id b$.
In particular, since $\cond l$ is prime to $\id{dN}$, we have 
that $\Delta(\zeta x)$ is $\cond l$-integral. Since $\cond l$ is square-free, it
remains to prove that $\pv \mid \Delta(\zeta\p x \p)$ for every $\pv \mid \cond l$.

Suppose $\pv \nmid \Delta(\zeta\p x\p)$ for some $\pv\mid\cond l$.
Since $\pv\nmid\id{d}$, we have that $\zeta\p x\p\in
L\p^\sharp$, and since $\pv\nmid\id N$ we have that
$L\p^\sharp=L\p$. 
We can then use \cite[Lemma 2.5]{tornatesis}
(which requires $\pv$ to be odd and $L\p=L\p^\sharp$)
to produce an element $z\p \in B\p^\times$ such that conjugation by $z\p$ fixes
both $L\p$ and $\zeta\p x\p$, and such that $\chic[\pv]{l}(\norm(z\p)) = -1$.
In particular, by \eqref{eqn:wtransp} we have that $\eta(z\p^{-1} y 
z\p) = -\eta(y)$ for every $y \in \widehat{F}^3$. Hence by 
\eqref{eqn:weil_fourier} we obtain that
\begin{align*}
{^{\iota}\eta}(x) = \, &
{^{\iota}\eta}(z\p^{-1} x z\p) =
c_S \int_{\widehat{F}^3} \eta(y) e_{\mathbf f} (-(z\p^{-1} x z\p)^t S y) \, dy \\ 
= \, & c_S \int_{\widehat{F}^3} \eta(y) e_{\mathbf f} (-x^t S (z\p y z\p^{-1}))
\,dy
= c_S \int_{\widehat{F}^3}\eta(z\p^{-1} y z\p) e_{\mathbf f} (-x^t S y)) \, dy\\ 
= & - c_S \int_{\widehat{F}^3} \eta(y) e_{\mathbf f} (-x^t S y)) \, dy
= - {^{\iota}\eta}(x).
\end{align*}
This contradicts ${^{\iota}\eta}(x) \neq 0$, hence $\pv \mid \Delta(\zeta\p x\p)$, 
which completes the proof.
\end{proof}

The following result is a modification of \cite[Proposition 11.7]{shim-hh},
adapted for our purposes.

\begin{prop}\label{prop:eta_modularity}

The function $\eta$ satisfies that
\begin{equation}\label{eqn:weil_gral}
	{^{\beta}\eta} = \chic[\cond u]{u}(a_\beta) \,
	\eta \qquad \forall\, \beta \in \Gamma[4\id{N}].
\end{equation}

\end{prop}

\begin{proof}

Denote
\[
P = \left\{\beta \in \SL_2(F):a_\beta = d_\beta = 1, c_\beta =
0\right\}.
\]
By \cite[Lemma 3.4]{shim-hh}, for every integral ideal $\id e$ 
such that $\id e \subseteq (2\id{d}^{-1}) \cap (2^{-1} \id{d} \cdot 
4\id{N})$, the group $\Gamma[4\id{N}]$ is generated by its intersection with
$P$, its intersection with $\iota P \iota^{-1}$, and $\Gamma[\id e, \id e]$.

Let $\beta \in \Gamma[4\id{N}]$.
If $\beta \in P$, then \eqref{eqn:weil_gral} follows by combining
\eqref{eqn:weil_parab} and part (b) of Lemma~\ref{lem:weil_casos}, since
$b_\beta \in 2 \id d^{-1}$.
If $\beta \in \iota P \iota^{-1}$, then \eqref{eqn:weil_gral} follows by
combining \eqref{eqn:weil_parab} (with ${^{\iota}\eta}$ instead of $\eta$) and
part (d) of Lemma~\ref{lem:weil_casos}, since $c_\beta \in 2 \id d \cdot \id N$.

Let $U$ be as in Proposition~\ref{prop:shimU}. We can assume that $U =
\prod_{\pv \in \mathbf{f}} U\p$ with $U\p = \{\beta \in \GL_2({\OO}\p) : \beta
\equiv I \mod \pi\p^{r\p} M_2({\OO}\p)\}$, where the $r\p$ are nonnegative
integers such that $r\p = 0$ for almost all $\pv$.
Let $\id e \subseteq \OO$ be such that $r\p = v\p(\id e)$.
Changing $\id e$ by a smaller ideal, we can assume that
$\id e \subseteq 2 \id{d} \id N \cond{lu}$.
In particular, $\id e$ satisfies the hypotheses mentioned above.

Let $\beta \in \Gamma[\id e, \id e]$. 
Since $\beta \in \Gamma[\id{d}^{-1},\id{d}]$, we have that $\id{a}_\beta =
\OO$. 
Furthermore, we have that $\beta \in \Gamma[\cond{lu}]$, which implies that
$a_\beta d_\beta \equiv 1 \mod \cond{lu}$.
Hence
\begin{equation}\label{eqn:caracteres}
\chic[\mathbf{a}]{lu}(d_\beta) \,
\chic[*]{lu}(d_\beta \id{a}_\beta^{-1}) = \chic[\cond{lu}]{lu}(d_\beta) =
\chic[\cond{lu}]{lu}(a_\beta). 
\end{equation}
Let $\pv \in \mathbf{f}$.  If $\pv$ is such that $r\p > 0$ take $\xi\p = 
a_\beta$. Else, take $\xi\p = 1$. 
Since $a_\beta d_\beta \equiv 1 \mod \id \pi\p^{r\p}$, we have that
\[
\beta\p = 
\left(\begin{matrix} \xi\p & 0\\ 0&  \xi\p^{-1}\end{matrix}\right)
\left(I + \left(\begin{matrix} 0 & \xi\p^{-1}b_\beta  \\
\xi\p c_\beta &  \xi\p d_\beta - 1\end{matrix}\right)\right) \quad \in
\left(\begin{matrix} \xi\p & 0\\ 0&  \xi\p^{-1}\end{matrix}\right) U\p.
\]  
Combining this with Proposition~\ref{prop:shimU}, part (a) of 
Lemma~\ref{lem:weil_casos} and \eqref{eqn:caracteres} we get that
\[
	{^{\beta}\eta}(x) =
	\chic[\cond{lu}]{lu} (a_\beta) \, \eta(\xi x) = 
	\chic[\cond{lu}]{lu} (a_\beta) \, \chic[\cond{l}]{l} (a_\beta) \, \eta(x).
\]
Since $\cond l$ is prime to $\cond u$ we have that $\cond{lu} = \cond l \cond
u$. Furthermore, we have that $\chic[\cond l]{u} = 1$ and $\chic[\cond
u]{l}(a_\beta) = 1$, since $a_\beta \in \OOx[\pv]$ for $\pv\mid \cond u$.
These facts imply that $\chic[\cond{lu}]{lu} (a_\beta) \chic[\cond{l}]{l} (a_\beta)
= \chic[\cond{u}]{u} (a_\beta)$,
which completes the proof.
\end{proof}

\begin{lemma}
	Let $\pv \mid \cond l$.
	Let $\nu \in \mathcal S(F\p^3)$ be such that
	\begin{enumerate}
		\item $\nu(\xi x) = \chic[\pv]{l}(\xi) \nu(x)$ for every $\xi \in \OOx[\pv]$.
		\item $\nu(x) = 0$ unless $\Delta(x)/l \in \OO[\pv]$.
	\end{enumerate}
	Given $b \in F \cap \OO[\pv]$ and $\beta\in \SL_2(F)$ with $c_\beta = 0$,
	let $\gamma = \beta \iota \smat{1 & b \\ 0 & 1}$.
	Then ${^{\gamma}}{\nu}(0) = 0$.
\end{lemma}

\begin{proof}
    By \eqref{eqn:weil_parab} and Proposition~\ref{prop:shimU} we can assume
	that $\beta = 1$. Using \eqref{eqn:weil_parab} and \eqref{eqn:weil_fourier}
	we have that	
	\[
		{^\gamma}\nu(0) = c_S \int_{F\p^3} e\p(x^t S x \, b/2) \nu(x) \,dx.
	\]
	Then the hypotheses on $\nu$ imply that for every $\xi \in \OOx[\pv]$ we have
	that
	\[
		{^\gamma}\nu(0) 
		 = c_S  \int_{F\p^3} \nu(x) \,dx
		 = \chic[\pv]{l}(\xi) \, c_S \int_{F\p^3} \nu(x) \,dx.
	\]
	Hence the result follows, since $\chic[\pv]{l}$ is non-trivial on $\OOx[\pv]$.
\end{proof}

\begin{lemma}\label{lem:cusp_eta}

	Assume that $l \notin \Fc$. Let $\beta \in \SL_2(F)$.
	Then ${^\beta}\eta(0) = 0$.

\end{lemma}

\begin{proof}
	Let $\beta = \smat{a & b \\ c & d}$.
	We have that ${^\beta}\eta(0) = \prod_{\pv \in \mbf} {^\beta}\eta\p(0)$.
	Since $l \notin \Fc$, there exists $\pv$ such that $\pv \mid \cond l$.
	Hence it suffices to prove that ${^\beta}\eta\p(0) = 0$.

	Write
	\[
	\beta =
	 \begin{cases}
         \phantom{-}
		 \pmat{ 1/c & a \\ 0 & c }
		 \iota
         \pmat{1& \phantom{-}d/c \\ 0 & 1} &
         \text{if $d/c \in \OO[\pv]$,} \\
		 - \pmat{ 1/d & b \\ 0 & d }
		 \iota
		 \pmat{1& -c/d \\ 0 & 1} \iota \quad &
         \text{if $c/d \in \OO[\pv]$.} \\
	 \end{cases}
	\]
	Then the result follows by applying the previous lemma to $\nu = \eta\p$ if
	$d/c \in \OO[\pv]$, and to $\nu = {^\iota}\eta\p$ if $c/d \in \OO[\pv]$.
	In the first case, the hypotheses needed on $\nu$ hold by parts (a) and (b)
	of Lemma~\ref{lem:weil_casos}.
	In the second case, using that $\cond l$ is prime to $\id d \id N$, they
	hold by parts (c) and (d) of Lemma~\ref{lem:weil_casos}.
\end{proof}

\begin{proof}[Proof of Proposition~\ref{prop:modularity}]

For each $\tau\in\mathbf{a}$ write $S_\tau = 
A_\tau^t\,A_\tau$, with $A_\tau\in\GL_3(\R)$. 
Then there exist homogeneous harmonic polynomials 
$Q_\tau(X)$ of degree $k_\tau$ such that 
\[
P(y) = \prod_{\tau\in\mathbf{a}} Q_\tau(A_\tau\,y_\tau)\,.
\]
For each $\tau$ we may assume that $Q_\tau(X) = (z_\tau^t \,X)^{k_\tau}$, with
$z_\tau \in \C^3$ such that $z_\tau^t \,z_\tau = 0$ (see \cite[Theorem
9.1]{iwaniec-topics}). Let $w_\tau = (A_\tau^{-1}\, z_\tau)^t$. 
Then we have that $w_\tau^t\,S_\tau\,w_\tau = 0$. Let $\sigma:F_\mba^3\to\C$ 
be the function given by $\sigma(x) = \prod_{\tau\in\mathbf{a}} 
(w_\tau^t\,S_\tau\,x_\tau)^{k_\tau}$, so that $P(y) =
\sigma(y)$. Then we have that $\tlu(z)=f(z;\eta)$,
where $f(z;\eta)$ is the function given by
\[
f(z;\eta) 
= \sum_{x\in F^3}\eta(x)\,\sigma(x)\,e_F(x^t\,S\,x, z/2)\,.
\]

The modularity of $\tlu$ follows by combining 
Proposition~\ref{prop:eta_modularity} and \cite[Proposition 
11.8]{shim-hh}, since the latter claims that
\begin{equation}\label{eqn:shim_automorph}
f\left(\beta z; {{^{\beta}\eta}}\right)
= \chic[\cond{-1}]{(-1)}(a_\beta) K(\beta,z) f(z;\eta)
\end{equation}
for every $\beta$ in (the certain subgroup of) $\SL_2(F)$,
where we denote by $K(\beta,z)$ the automorphy factor of weight
$\mathbf{3/2+k}$.

To prove \eqref{eqn:cfourier}, we let $U$ be as in Proposition~\ref{prop:shimU}.
Assume that $U$ is small enough so that $d_{\beta\p} \equiv 1 \mod \cond{lu}
\OO[\pv]$ for every $\beta \in U$, for every $\pv \mid \cond{lu}$,
and so that $U \subseteq \widehat \Gamma[2\id{d}^{-1},2 \id{d} \cdot \id N]$.
Let $\xi \in \hatx{F}$ be such that $\xi\hatOO \cap F = \id a$,
and pick $\beta \in \SL_2(F)\cap\left(\begin{smallmatrix} \xi & 0\\ 0 &
\xi^{-1}\end{smallmatrix}\right) U$.
It is easy to verify that $\id{a}_\beta = \id{a}^{-1}$. Furthermore $d_\beta
\id a$ is prime to $\cond{lu}$. Hence since $\xi\p d_\beta \equiv 1 \mod
\cond{lu} \OO[\pv]$ for every $\pv \mid \cond{lu}$ we get that
\begin{equation*}
	\chic[\mathbf{a}]{lu} (d_\beta) \, \chic[*]{lu}(d_\beta \id{a}_\beta^{-1})  = 
	\chic[\cond{lu}]{lu} (d_\beta) \prod_{\pv \nmid \cond{lu}} \chic[\pv]{lu}(\xi\p) =
	\chic{lu}(\xi).
\end{equation*}
Similarly, we have that
$\chic[\mathbf{a}]{-u} (d_\beta) \chic[*]{-u}(d_\beta \id{a}_\beta^{-1}) = 
\chic{-u}(\xi)$.
This equalities, combined with \cite[3.14c]{shim-hh} (which requires the second
condition we imposed on $U$), Proposition~\ref{prop:shimU} 
and \eqref{eqn:shim_automorph}, imply that
\begin{multline*}
  \idn{a}^{-1/2} \sum_{D\in F} \lambda(-uD,\id{a};\tlu) \, e_F(-uD,z/2)
  = \chic{-u}(\xi) K(\beta, \beta^{-1} z) \,
  f\left(\beta^{-1}z; \eta\right) \\
  = \chic{u}(\xi) f\left(z; {{^{\beta}\eta}}\right)
  = \chic{l}(\xi) \, \idn{a}^{-3/2}
  \sum_{x\in F^3}\eta(\xi x)\,\sigma(x)\,e_F(x^t\,S\,x, z/2).
\end{multline*}
This proves that $\lambda(-uD,\id a; \tlu) = c(D,\id a;w,P)$,
since $\chic{l}(\xi) \eta(\xi x) = w(x;\id{ab})$.

Finally, to prove the cuspidality, we let $\beta \in \SL_2(F)$ and
\[
	\theta(z) = \chic[\cond{-1}]{(-1)}(a_{\beta^{-1}}) K(\beta^{-1},z)^{-1} \,
	\tlu(\beta^{-1} z).
\]
By \eqref{eqn:shim_automorph} we have that $\theta(z) = f(z;{^{\beta}}\eta)$.
Then $\lim_{z \to i \infty} \theta(z) = P(0) \, {^{\beta}}\eta(0)$.
If $\mbk \neq \mathbf 0$, then $P(0) = 0$. Otherwise we have that $l \notin
\Fc$, and then Lemma~\ref{lem:cusp_eta} implies that ${^{\beta}}\eta(0) = 0$.
This proves that $\tlu$ vanishes at $\beta^{-1}\cdot \infty$.
\end{proof}

\bigskip

Assume for the rest of this section that $\sgn(u) = (-1)^\mbk$.
We consider the \emph{Kohnen plus subspace}, defined by
\begin{multline*}
	\Shk[+] = \Bigl\{ f \in \Shk \; :\\
		\lambda(-uD,\id a;f) = 0 \text{ unless 
	$(D,\id a)$  is a discriminant} \Bigr\}.
\end{multline*}
This agrees with the definition given by \cite{hiraga-ikeda} and \cite{su},
taking into account the different normalizations for the automorphy factor.

\begin{prop}\label{prop:cuspidality}
	Let $\tlu$ be as above. Assume that $l \notin \Fc$ or $\mbk \neq
	\mathbf 0$.
	Then $\tlu \in \Shk[+]$.
\end{prop}

\begin{proof}

Assume that $\lambda(-uD,\id a; \tlu) \neq 0$.
Then $D \neq 0$, since $\mathcal A_{0, \id c}(L) = \{0\}$ and $w(0;\id c) \,
P(0) = 0$.
Furthermore, by \eqref{eqn:cfourier} we have that
$\mathcal{A}_{\Delta, \id c}(L) \neq \varnothing$,
which implies that $(\Delta, \id c)$ is a discriminant. Since $D \in \id
a^{-2}$, the following result implies that $(D, \id a)$ is a discriminant.
\end{proof}

\begin{prop}\label{prop:discvsdisc}

	Given $D \in F^{-l}$
    and $\id a \in \IF$ such that $D \in \id a^{-2}$,
	let $\Delta = lD$ and $\id c = \id{ab}$. Then $(D,\id a)$ is a discriminant
	if and only if $(\Delta, \id c)$ is a discriminant. 
	Furthermore, if $\cond D$ is prime to $\cond l$, then $(D,\id a)$ is
	fundamental if and only if $(\Delta,\id c)$ is fundamental.

\end{prop}

\begin{proof}

	We can assume that $l,D \notin \Fc$, since otherwise the result is trivial.

	Let $\omega' \in F(\sqrt l)$ be such that $\Delta(\omega') = l$ and $\OO
	\oplus \omega' \id b$ is an order. Assume that $(D, \id a)$ is a
	discriminant. Then there exists $\omega \in F(\sqrt D)$ such that
	$\Delta(\omega) = D$ and $\OO \oplus \omega \id a$ is an order. 
	Let $\omega'' = \omega \overline{\omega'} + \omega' \overline{\omega}$. A
	simple calculation shows that
	\[
		(\omega'')^2 = \trace(\omega) \trace(\omega') \omega''
		+ \big(4 \norm(\omega)\norm(\omega')
	- \trace(\omega')^2 \norm(\omega) - \trace(\omega)^2 \norm(\omega')\big),
   \]
   which implies that $\Delta(\omega'') = \Delta$.
   Furthermore, using that $\trace \omega \in \id a^{-1}$ and $\norm \omega \in
   \id a^{-2}$, and that $\trace \omega' \in \id b^{-1}$ and $\norm \omega' \in
   \id b^{-2}$, we see that $\trace \omega'' \in \id c^{-1}$ and $\norm \omega''
   \in \id c^{-2}$. 
   This implies that $\OO \oplus \omega'' \id c$ is an order, and
   hence that $(\Delta, \id c)$ is a discriminant. 

   If $(\Delta, \id c)$ is a discriminant, then $(l^2 D, \id b^2 \id a)$ is a
   discriminant, and hence $(D, \cond l \id a)$ is a discriminant. Since
   $D \in \id a^{-2} $ and $\cond l$ is odd, by \cite[Proposition
   2.12]{waldspu} we can conclude that $(D, \id a)$ is a discriminant.

   Finally, if $\cond D$ is prime to $\cond l$ then $\cond \Delta = \cond l
   \cond D$. Then $\cond D$ is square-free if and only if $\cond \Delta$ is
   square-free, which proves the last assertion.
\end{proof}

\section{The theta map and special points}\label{sect:special_pts}

In this section we consider the representation of $B^\times /F^\times$ into
the space $V_\mathbf{k}$ of homogeneous polynomials in $W = B/F$ of
degree $\mathbf{k}$, harmonic with respect to $\Delta$.
We denote the corresponding spaces of quaternionic modular forms by 
$\mathcal{M}_\mathbf{k}(R)$, etc.

For each $x \in \hatx B$ let $L_x \subseteq W$ be the lattice given by
$L_x = R_x / \OO$, and denote $L = L_1$.
Let $\id N$ be the level of $L$. Note that $L_x$ is in the same genus
as $L$, hence it also has level $\id N$ for every $x$.
As in Sect.~\ref{sect:half_int}, we fix $l \in F^\times$ and a nonzero weight
function $w$ on $\widehat L$ defined in terms of local weight functions on $L\p$ for
$\pv \mid \cond l$.
For each $x \in \hatx B$ we consider the weight function on $\widehat{L_x}$
given by
\begin{align*}
	w_x :  \What & \longrightarrow \{0,1,-1\} \\
             y & \longmapsto \chic{l} (\norm x) \, w(x y x^{-1}).
\end{align*}
Then by \eqref{eqn:wtransp} we have that
\begin{equation}\label{eqn:transf_wx}
 w_{z x \gamma}(\gamma^{-1} y \gamma) = w_x(y)
 \qquad \forall\, z \in \hatx R, \, \gamma \in B^\times.
\end{equation}
For $P \in V_\mathbf{k}$ we consider the theta series $\tlu_{x,P} =
\tlu(\vardot;w_x,P)$ given
by \eqref{eqn:theta}.
Using \eqref{eqn:transf_wx} we get that
\begin{equation}
	\tlu_{z x\gamma,P\cdot \gamma} = \tlu_{x,P}\,, \qquad\quad
	\forall\, z \in \hatx R,\, \gamma\in B^\times. \label{eqn:theta_transf}
\end{equation}

The theta series $\tlu_{x,P}$ define a linear map
$\tlu : \mathcal{M}_\mathbf{k}(R) \to \Mhk$
given by
\[
 \tlu(\varphi) = \sum_{x\in \cl(R)} \tfrac{1}{t_x} 
\tlu_{x,\varphi(x)}\,.
\]
This map is well defined by \eqref{eqn:phi=sum_phix} and 
\eqref{eqn:theta_transf}, and satisfies that $\tlu(\varphi_{x,P}) = 
\tlu_{x,P}$ for every $x \in \hatx B$ and $P \in V_\mathbf{k}$.

\medskip

We will now prove the Hecke-linearity of this construction.
We normalize the $\id p$-th Hecke operator acting on $\Mhk$ defined
in \cite{shim-hh}, multiplying it by $\idn p$.
The following auxiliary result is similar to \cite[Lemma 4.8]{yo}. 

\begin{lemma}
	Let $\id p$ be a prime ideal and let $x \in \hatx B$. Then $L_{hx} \subseteq
	\id p^{-1} L_x$ for every $h \in H\subp$. Furthermore, if $\id p\nmid 2 \id N$,
	for every $\id c \in \IF$ and $y\in (\id{pc})^{-1} L_x$ such that $\Delta(y)
	\in \id c^{-2}$, we have that
\[
    \numset{h\in \hatx{R}\backslash H\subp \;:\;
    y\in \id c^{-1}L_{hx} }
= \begin{cases}
1+\norm(\id p),& y\in\id p \id c^{-1}L_x,\\
1+\kro{\Delta(y) \zeta\subp^2}{\id p},& 
y\in \id c^{-1}L_x \backslash\id p \id c^{-1}L_x,\\
1,& y\in (\id{pc})^{-1} L_x \backslash \id c^{-1}L_x.
\end{cases}
\]
Here $\zeta\subp \in F\subp^\times$ is such that $c\subp = \zeta\subp \OO[\id
p]$.
\end{lemma}

\begin{prop}\label{prop:hecke_linear}
	The map $\tlu$ is Hecke-linear on prime ideals $\id p$ such that 
	$\id p \nmid 2\id N \cond l$.
\end{prop}

\begin{proof}
Let $x\in\hatx B$, $P\in V_{\mathbf{k}}$, and denote
$f = \tlu_{x,P}$ and $f' = \tlu(T\subp(\varphi_{x,P}))$.
Let $\id p$ be a prime ideal.
There is a bijection between $H\subp/\hatx{R}$ and $\hatx{R}\backslash H\subp$
given by $k \mapsto h = \pip k^{-1}$, under which we have
$L_{k^{-1}x} = L_{hx}$. Using this and Proposition~\ref{prop:hecke_on_phi},
\[
f' =
\sum_{k \in H\subp/\hatx{R}}\tlu_{k^{-1}x, P}
  = \sum_{h\in \hatx{R}\backslash H\subp}\tlu_{hx, P}.
\]

Let $D\in F^{-l} \cup \{0\}$ and $\id a \in \IF$.
Let $\Delta = lD$ and $\id c = \id {ab}$.
By $\eqref{eqn:wtransp}$ we have that
$w_{hx}(y; \id c) = w_x(y; \id{pc})$
for every $h \in H\p$ and $y \in \id c^{-1} L_{hx}$.
Using this, \eqref{eqn:fourier_coeffs} and \eqref{eqn:cfourier} we have that
\begin{align*}
 \lambda(-uD, \id a; f')&=\sum_{h\in \hatx{R}\backslash H\subp}
 c(D, \id a;w_{hx},P)\\
 &=\frac1{\norm(\id a)}\sum_{h\in \hatx{R}\backslash H\subp}
 \sum_{y\in\mathcal{A}_{\Delta, \id c}(L_{hx})} w_{hx}(y;\id c)\, P(y)\\
 &=\frac{1}{\norm(\id a)}\sum_{(h, y)\in\Lambda} w_x(y; \id{pc}) \, P(y) \,
\end{align*}
where $\Lambda = \left\{(h, y)\in \left(\hatx{R}\backslash H\subp \right)
	\times W \; : \; y\in \mathcal A_{\Delta, \id c}(L_{hx})\right\}$. 

We consider the decomposition $\Lambda = \Lambda_1\cup\Lambda_2\cup\Lambda_3$,
where
\begin{align*}
\Lambda_1 &= \{(h, y)\in\Lambda\; : \;y\in\id p\id c^{-1}L_x\},\\
\Lambda_2 &= \{(h, y)\in\Lambda\; : \;y\in \id c^{-1}L_x\backslash
\id p\id c^{-1}L_x\},\\
\Lambda_3 &= \{(h, y)\in\Lambda\; : \;y\in(\id{pc})^{-1}L_x\backslash
\id c^{-1}L_x\}.
\end{align*}
Using \eqref{eqn:wsegda}, \eqref{eqn:fourier_coeffs} and the previous lemma
we see that
\begin{align*}
	\frac{1}{\norm(\id a)} & \sum_{(h,y)\in\Lambda_i}
	w_{x}(y;\id{pc}) \, P(y) = \\
	 & \begin{cases}
	(1+\norm(\id p))\frac{c(D,\id p^{-1}\id a;w_x,P)}{\norm(\id p)}, & i = 1,\\
	\left(1+\kro{\Delta \zeta\subp^2}{\id p}\right)
	\left(\chic[*]{l}(\id p) c(D,\id a;w_x,P) - 
	\frac{c(D,\id p^{-1}\id a;w_x,P)}{\norm(\id p)}\right), & i = 2,\\
	\norm(\id p) c(D,\id p\id a;w_x,P) - \chic[*]{l}(\id p) c(D,\id a;w_x,P), & i = 3.
	\end{cases}
\end{align*}
We observe that $\kro{\Delta \zeta\subp^2}{\id p} c(D,\id p^{-1}\id a;w_x,P) = 0$,
because if $c(D,\id p^{-1}\id a;w_x,P) \neq 0$ then there exists
$y\in \mathcal A_{\Delta, \id p^{-1}\id c}(L_x)$, which implies that $\id p\mid
\Delta(y)\zeta\subp^2 =\Delta\zeta\subp^2$.
Furthermore, if $\xi\subp \in F\subp^\times$ is such that $\id a\subp =
\xi\subp \OO[\id p]$, we have that
$\kro{\Delta \zeta\subp^2}{\id p} = \chic[*]{l}(\id p) \kro{D\xi\subp^2}{\id p}$.
Using this, \eqref{eqn:cfourier} and adding up we obtain that
\begin{multline*}
 \lambda(-uD, \id a; f')
  = \lambda(-uD,\id p^{-1}\id a;f) \\
  + \kro{D\xi\subp^2}{\id p}\lambda(-uD,\id a;f)
 + \norm(\id p)\lambda(-uD,\id p\id a;f) \,.
\end{multline*}
Writing $\kro{D\xi\subp^2}{\id p} = \chic[*]{-u}(\id p)
\kro{-uD\xi\subp^2}{\id p}$ and using
\cite[Proposition 5.4]{shim-hh} (and recalling our normalization), we see that
$\lambda(-uD,\id a;f') = \lambda(-uD,\id a; T\subp(f))$, which completes the
proof.
\end{proof}
 
\subsection*{Special points}

Let $D \in F^{-l}$ and $\id a \in \IF$ be such that $(D,\id a)$ is a
discriminant.
Consider the discriminant $(\Delta, \id c) = (lD, \id{ab})$, let $K = F(\sqrt
\Delta)$ and let $\omega \in K$ be such that $\Delta(\omega) = \Delta$ and
$\OO[\Delta,\id c] = \OO \oplus \omega \id c$ is an order in $K$.

Furthermore, assume that there exists an embedding $\OO[\Delta,\id c]
\hookrightarrow R$. Then we can consider the set $\widetilde X_{\Delta,\id{c}} =
\{x\in \hatx B\,:\, \OO[\Delta,\id{c}] \subseteq R_x\}$, and define a set
$X_{\Delta,\id{c}}$ of \emph{special points} associated to the discriminant
$(\Delta,\id c)$ by
\[
 X_{\Delta, \id{c}} = \hatx R\backslash\widetilde X_{\Delta, \id{c}} /K^\times.
\]

We let $P_\Delta \in V_\mathbf{k}$ be the polynomial characterized by the property
\begin{equation}\label{eqn:gegenbauer}
  w(\omega;\id c) \, P (\omega)
  = \ll{P, P_\Delta} \qquad \forall \, P\in V_\mathbf{k}\,,
\end{equation}
Note that $\omega \in K/F$ is uniquely determined up to sign; assuming $\sgn(l)
= (-1)^\mbk$ (see Hypothesis \ref{hyp:peso_signo} below) we have that
$w(\vardot;\id c) \, P(\vardot)$ is an even function, and hence $P_\Delta$ does
not depend on $\omega$.
Since $(P \cdot a) (\omega) = P(\omega)$ for every $a\in K^\times$, we have that
\begin{equation}\label{eqn:P_D-Kinv}
 P_\Delta \cdot a = P_\Delta
 \qquad \forall \, a \in K^\times / F^\times. 
\end{equation}

\begin{lemma}\label{lem:womeganeq0}
  
	If $\cond D$ is prime to $\cond l$, we have that $w(\omega;\id c) \neq 0$.

\end{lemma}

\begin{proof}

Let $\xi \in \hatx F$ be such that $\xi \hatOO \cap F = \id c$. Let $\pv
\mid \cond l$. Then
\[
  \Delta(\xi\p \omega) \OO[\pv]
  = \id c^2 \Delta \OO[\pv] = \cond l \cond D \OO[\pv]
  = \cond l \OO[\pv] \,.
\]
In particular, since $\cond l$ is square-free we have that $\xi\p \omega \notin
\cond l L\p$. Hence by \eqref{eqn:wvector} we have that $w(\omega;\id c) =
\chic{l}(\xi) \prod_{\pv \mid \cond l} w\p(\xi\p \omega)$ is not zero.
\end{proof}

For the rest of this section assume that $\cond D$ is prime to $\cond l$.
We let
\begin{equation}\notag
 \eta\ssl_{D, \id{a}} = \tfrac1{w(\omega;\id c)}
 \sum_{x \in X_{\Delta, \id{c}}} 
 \frac{w_x(\omega;\id c)}{[\OOx[x]:\OOx]}\:
 \varphi_{x,P_{\Delta}} \qquad \in \mathcal{M}_\mathbf{k}(R)\,,
\end{equation}
where $\OO[x] = R_x \cap K$. This is well defined by \eqref{eqn:transf_phi2}, 
\eqref{eqn:transf_wx} and \eqref{eqn:P_D-Kinv}. 
When $(D,\id a)$ is fundamental we have that 
\begin{equation}\label{eqn:eta_fund}
\eta\ssl_{D,\id a} 
 = \tfrac1{w(\omega;\id c) \, t_K}
 \sum_{x \in X_{\Delta, \id{c}}} w_x(\omega;\id c) \varphi_{x,P_{\Delta}},
\end{equation}
because, since $(\Delta, \id c)$ is fundamental, we have that $\OO[x]=\OO[K]$
for every $x\in X_{\Delta,\id{c}}$.
It can be proved that $\eta\ssl_{D, \id a}$ does not depend on the choice of 
the embedding $\OO[\Delta, \id c] \hookrightarrow R$. 

Finally, if $(D,\id a)$ is not a discriminant or if there does not exist an
embedding $\OO[\Delta, \id c] \hookrightarrow  R$ we let $\eta\ssl_{D,\id a}=0$.
 
\begin{prop}\label{prop:coef_serie_theta}
Let $\varphi\in \mathcal M_\mathbf{k}(R)$.
Let $D\in F^{-l}$ and let $\id a \in \IF$. 
Then
\begin{equation}\label{eqn:fourier=height}
 \lambda(-uD,\id a;\tlu(\varphi)) =
 \frac1{\idn a}\,\ll{\varphi, \eta\ssl_{D,\id a}}\,.
\end{equation}
\end{prop}

\begin{proof}

We can assume that $(D, \id a)$ is discriminant and that there exists an
embedding $\OO[\Delta,\id c] \hookrightarrow R$ since otherwise, by
Proposition~\ref{prop:discvsdisc}, both sides of \eqref{eqn:fourier=height}
vanish.
Furthermore, by \eqref{eqn:phi=sum_phix} we can assume that $\varphi =
\varphi_{x,P}$ with $P\in V_\mathbf{k}^{\Gamma_{\!x}}$, so that
$\tlu(\varphi) = \tlu_{x,P}$. 

Let $\Gamma_{\!x}$ act on $\mathcal{A}_{\Delta,\id{c}}(L_x)$ by conjugation.
Given $y \in \mathcal{A}_{\Delta,\id{c}}(L_x)$, since $\disc(y)=\disc(\omega)$,
we can assume there exists $\gamma\in B^\times$ such 
that $y=\gamma\omega\gamma^{-1}$. In particular, $\OO[\Delta,\id{c}] = \OO 
\oplus 
\id{c}\,\omega$. The map $y\mapsto x\gamma$ induces an injection
\[
	\Gamma_{\!x} \backslash \mathcal{A}_{\Delta,\id{c}}(L_x)
	\longrightarrow
	X_{\Delta,\id{c}}\,.
\]
Note that
$\Stab_{\,\Gamma_{\!x}} y = (R_x\cap F(y))/\OOx
\simeq
(R_{x\gamma}\cap K)/\OOx = \OOx[x\gamma]/\OOx$.
Note that
$w(\omega;\id c) \, w_x(y;\id c) \, P(y)
= w_{x\gamma}(\omega;\id c) \, \ll{P\cdot\gamma,P_\Delta}$.
Furthermore, using that $P$ is fixed by $\Gamma_x$, by 
Proposition~\ref{prop:height_on_phi} we have that
$\ll{P\cdot\gamma,P_\Delta} = \frac1{t_x} \ll{\varphi_{x,P},
\varphi_{x\gamma,P_\Delta}}$. Using these facts, \eqref{eqn:fourier_coeffs}
and \eqref{eqn:cfourier} we get that
\begin{align*}
\idn a \,\lambda(-uD,\id a;\tlu_{x,P})
& = \sum_{y\in\mathcal{A}_{\Delta,\id{c}}(L_x)} w_x(y;\id c) \, P(y) \\
& = \sum_{y\in\Gamma_{\!x}\backslash\mathcal{A}_{\Delta,\id{c}}(L_x)}
[\Gamma_{\!x}:\operatorname{Stab}_{\,\Gamma_{\!x}}y]  \, w_x(y;\id c) \, P(y)\\
& = \tfrac1{w(\omega;\id c)}\sum_{x\gamma\in X_{\Delta,\id{c}}}
\frac{t_x \, w_{x\gamma}(\omega;\id c)}{[\OOx[x\gamma]:\OOx]}\,
\ll{P\cdot\gamma,P_\Delta}
\\&= \tfrac1{w(\omega;\id c)} \sum_{z\in X_{\Delta,\id{c}}}
\frac{w_z(\omega;\id c)}{[\OOx[z]:\OOx]}\,
\ll{\varphi_{x,P}, \varphi_{z,P_\Delta}}
 = \ll{\varphi_{x,P},\eta\ssl_{D, \id a}} \,.
\end{align*}
Note that in the last sum $\ll{\varphi_{x,P}, \varphi_{z,P_\Delta}}=0$
unless $z=x\gamma$.
\end{proof}

\section{Central values and the height pairing}\label{sect:height_geom}

We start this section by comparing the geometric pairing on CM-cycles of 
\cite{zhang-gl2} (see \cite{Xue-Rankin} for the case of higher 
weight) with the height pairing introduced in Sect.~\ref{sect:quat_forms}.
We refer to \cite[Sect.~3]{waldspu} for details.

Let $K/F$ be a totally imaginary quadratic extension.
We assume that there exists an embedding $\OO[K] \hookrightarrow R$, which we
fix.
Let $\sC = (\hatx B / \hatx F) / (K^\times / 
F^\times)$, and let $\pi:\hatx B / \hatx F\to\sC$ be the projection 
map.
We fix a Haar measure $\mu$ on $\hatx B / \hatx F$.
We write $\muR = \mu\bigl(\hatx R/\hatOOx\bigr)$.

We consider the space $\DC$ of 
\emph{CM-cycles} on $\sC$. These are locally constant
functions on $\sC$ with compact support.
This space comes equipped with the action of 
Hecke operators $\Tm$.
Furthermore, given $v\in V$ which is fixed by $K^\times/F^\times$, 
we consider the \emph{geometric pairing} $\ll{\cdot, \cdot}_v$ on $\DC$ induced
by $v$ as in \cite{waldspu}.

Given $a \in \hatx K$, we let $\alpha_a \in \DC$ be the 
characteristic function of $\pi\bigl(\hatx Ra\bigr) \subseteq \sC$.
Since $\OO[K]\subseteq R$, the CM-cycle $\alpha_a$ depends only on the element
in $\cl(K)$ determined by $a$.
The same holds for the quaternionic modular form $\varphi_{a,v}$, by
\eqref{eqn:transf_phi2}.

\begin{prop}\label{prop:geom_vs_height}

Let $\id{m}\subseteq\OO$ be an ideal. For $a, b \in \cl(K)$ we have that
\[
 \frac{\ll{\Tm\alpha_a, \alpha_b}_v}{\muR}
 =
 \tfrac1{\tK^{\;2}}\,
 \sum_{\xi\in\cl(F)}\ll{\Tm\varphi_{\xi a,v}, \varphi_{b,v}}\,.
\]
\end{prop}

\begin{proof}
 See \cite[Proposition 3.5]{waldspu}.
\end{proof}

Let $\Xi$ be a character of $\cl(K)$. 
Let $\alpha_{\Xi}\in \DC$ be the CM-cycle given by
\begin{equation}\notag
    \alpha_{\Xi} = \tfrac{\mK}{\hF} \sum_{a\in\cl(K)} \Xi(a) \alpha_a\,.
\end{equation}
Similarly we define
\begin{align}\label{eqn:alpha_K}
    \psi_{\Xi,v} = \tfrac1{\tK} \sum_{a\in\cl(K)} \Xi(a) \varphi_{a,v}\qquad 
\in \mathcal{M}_\rho(R)\,.
\end{align}
After these definitions and Proposition~\ref{prop:geom_vs_height} we get the 
following result.

\begin{coro}\label{coro:geom=height}
Let $\id{m}\subseteq\OO$ be an ideal. 
Assume that $\Xi$ is trivial in $\cl(F)$.
Then $\psi_{\Xi,v} \in \mathcal{M}_\rho(R,\bbu)$ and 
\[
 \frac{\ll{\Tm\alpha_{\Xi}, \alpha_{\Xi}}_v}{\muR}
 = \frac{\mK^{\;2}}{\hF}
 \,\ll{\Tm\psi_{\Xi,v}, \psi_{\Xi,v}}\,.
\] 
\end{coro}

\subsection*{Central values}
Let $l \in F^\times$ and $D \in F^{-l}$.
Denote $\Delta = lD$, and let $K = F(\sqrt{\Delta})$. 
We let $\Xi$ be character of $C_K$ corresponding to the quadratic extension 
$F(\sqrt{l},\sqrt{D}) / K$. It satisfies that
\begin{equation}\label{eqn:genus_char}
  \Xi(a) = \chic{l}\left(\norm(a)\right)
         = \chic{D}\left(\norm(a)\right)
         \qquad \forall \, a \in C_K.
\end{equation}
If $\cond{l}$ and $\cond{D}$ are prime to each other, then 
$\cond{\Delta} = \cond{l} \cond{D}$, and \eqref{eqn:genus_char} implies that 
$\Xi$ is unramified; hence it is a character of $\cl(K)$. Furthermore, it is 
trivial in $\cl(F)$. These conditions on $\Xi$ are used in
Theorem~\ref{thm:xue} below.

Let $\id N \subseteq \OO$. We assume that $\cond{\Delta}$ is prime to $\id N$,
and that
\begin{equation}\label{eq:SigmaD}
    \Sigma_{\Delta} = \mathbf{a} \cup
    \left\{\id p\mid \id N \;:\; \chic[*]{\Delta}(\id p)^{v\p(\id N)}=-1 
\right\}
\end{equation}
is of even cardinality.
For the rest of this section we assume that $B$ is 
ramified exactly at $\Sigma_{\Delta}$.
Furthermore, we assume that $R$ has discriminant $\id N$.

Let $g \in \Sk$ be a normalized newform with trivial central character.
The Ranking-Selberg convolution $L$-function $L(s,g,\Xi)$ satisfies that
\[
  L(s,g,\Xi) = L(s,g\otimes\chic{l}) \, L(s,g\otimes\chic{D}) \,.
\]
We denote this function by $L_{l,D}(s,g)$.
A result of \cite{zhang-gl2}, in the case of parallel weight $\mathbf 2$, gives
the central value $L_{l,D}(1/2,g)$ in terms of the geometric pairing.
Using the generalization to higher weights given by Xue we obtain
the following formula.

\begin{thm}\label{thm:xue}
Let $\Tg$ be a polynomial in the Hecke operators prime to $\id N$ giving the 
$g$-isotypical projection.
Assume that $\id N\subsetneq\OO$ and $\cond{\Delta}$ is prime to $2 \id N$. 
Furthermore, assume that $\cond l$ and $\cond D$ are prime to each other.
Let $P_{\Delta}\in V_\mathbf{k}$ be as in \eqref{eqn:gegenbauer}. Then
\begin{equation}\notag
 L_{l,D}(1/2,g) =
 \ll{g,g}
 \:
 \frac{\dF^{\;1/2}}{\hF}
 \:
 \frac{c(\mathbf{k}) \: C(\id N)}{\mathcal N(\cond \Delta)^{1/2}}
 \:
 \frac{\mK^{\;2}}{(-\Delta)^\mathbf{k}}
 \:
 \ll{\Tg\psi_{\Xi,P_\Delta}, \psi_{\Xi,P_\Delta}}\,,
\end{equation}
where $\ll{g,g}$ is the Petersson norm of $g$, and $c(\mathbf{k})$ and $C(\id
N)$ are the positive rational constants given, respectively, in \cite[Corollary
3.13, Theorem 3.9]{waldspu}.
\end{thm}

\begin{proof}
Follows from \cite[Theorem 1.2]{Xue-Rankin}, Corollary~\ref{coro:geom=height}
and \cite[Lemma 3.14]{waldspu}.
\end{proof}

\section{An auxiliary result}\label{sect:orders}

Assume in this section that $R \subseteq B$ is an order of discriminant $\id N$ 
satisfying that for every $\id p \mid \id N$ the Eichler invariant $e(R\p)$ is 
not zero. The following result is Proposition 4.3 from \cite{waldspu}.

\begin{prop}\label{prop:bil_charact}
 
Let $\Bil(R\p) = R\p^\times \backslash N(R\p) / F\p^\times$. Then
\[
	\Bil(R\subp) \simeq
	\begin{cases}
		\Z/2\Z, & \id p \mid \id N, \\
		\{0\}, & \text{otherwise}.
	\end{cases}
\]

\end{prop}

Let $D \in F^{-l}$, and let $\id a \in \IF$ be such that $(D,\id a)$ is a
fundamental discriminant. 
Consider the fundamental discriminant $(\Delta, \id c) = (lD, \id{ab})$, and let
$K = F(\sqrt \Delta)$.
Since $\id a$ and $\id c$ are determined by $D$, we omit them in the subindexes.

As in Sect.~\ref{sect:height_geom}, we assume that there exists an embedding 
$\OO[K]\hookrightarrow R$, so that $1 \in \widetilde X_\Delta$. 
There is a left action of $\widetilde\Bil(R)$ on $X_\Delta$, 
induced by the action of $N(\widehat R)$ on $\widetilde X_\Delta$ by left 
multiplication. 
There is also a right action of $\cl(K) = \hatOOx[K] \backslash \hatx K / 
K^\times$ on $X_\Delta$, induced by the action of $\hatx K$ on $\widetilde X_\Delta$ by 
right multiplication.
 The following result is Proposition 4.7 from \cite{waldspu}.

\begin{prop}\label{prop:action_transitive}
The group $\cl(K)$ acts freely on $X_\Delta$, and the action of $\Bil(R)$ on
$X_\Delta / \cl(K)$ is transitive. Furthermore, the latter action is free if 
$(\cond{\Delta}:\id N) = 1$.
\end{prop}

Let $\delta\ssl$ be the character of $\Bil(R)$ given by
\begin{equation}\label{eqn:deltal}
	\delta\ssl(z) = \chic[*]{l}(\norm z).
\end{equation}
Assume that $\cond l$ and $\cond D$ are prime to each other, and let $\Xi$ be 
the character of $\cl(K)$ considered in the previous section.
Let $\etalD \in \mathcal{M}_\mathbf{k}(R)$ be as in \eqref{eqn:eta_fund},
and let $\psilD = \psi_{\Xi,P_\Delta} \in 
\mathcal{M}_\mathbf{k}(R,\bbu)$ be as in \eqref{eqn:alpha_K}. 

\begin{prop}\label{prop:eta=proy_psi}
 
Assume that $(\cond{\Delta}:\id N) = 1$. Then
\[
 \etalD = 
  \sum_{z \in \Bil(R)} \delta\ssl(z) \, \psilD \cdot z\,.
\]
In particular, $\eta\ssl_D \in 
\mathcal{M}_\mathbf{k}(R,\mathbbm{1})^{\delta\ssl}$.

\end{prop}

\begin{proof}

Let $a \in \hatx K$. Using \eqref{eqn:wtransp} and
Proposition~\ref{prop:bil_charact} we see that $\delta\ssl(z) w_a(\vardot; \id
c) = w_{z^{-1} a}(\vardot; \id c)$ for every $z \in \Bil(R)$.
Furthermore, using Lemma~\ref{lem:womeganeq0} we can write $w_a(\omega;\id c)
\chic{l}(\norm a) = w(\omega;\id c)$.
Using these facts, \eqref{eqn:transf_phi},
\eqref{eqn:genus_char} and Proposition~\ref{prop:action_transitive}, we get that
\begin{align*}
    \sum_{z \in \Bil(R)} \delta\ssl(z) \, \psilD \cdot z
& = \tfrac{1}{\tK}\sum_{z \in \Bil(R)} \delta\ssl(z) \sum_{a \in \cl(K)}
  \chic{l}(\norm a) \varphi_{z^{-1} a,P_\Delta} \\
& = \tfrac1{w(\omega;\id c) \tK}
  \sum_{z \in \Bil(R)} \sum_{a \in \cl(K)}
  \delta\ssl(z)  w_a(\omega;\id c)\varphi_{z^{-1}a,P_\Delta} \\
& = \tfrac1{w(\omega;\id c) \tK}\sum_{x \in X_\Delta} w_x(\omega;\id c)\varphi_{x,P_\Delta}
  = \etalD\,.
\end{align*}
\end{proof}

The following statement follows from this result and
Proposition~\ref{prop:bil_charact}.

\begin{coro}\label{coro:psi_vs_eta}
 Assume that $(\cond{\Delta}:\id N) = 1$. If $\varphi \in 
\mathcal{M}_\mathbf{k}(R,\mathbbm{1})^{\delta\ssl}$, then
\[
   \ll{\varphi,\etalD}
   = 2^{\omegaN}
   \ll{\varphi,\psilD}
   \,.
\]
\end{coro}

\section{Main theorem}\label{sect:main}

Let $g \in \Sk$ be a normalized newform with trivial central character as in the
introduction, with Atkin--Lehner eigenvalues $\varepsilon_g(\id p)$ for $\id p
\mid \id N$.

Fix $(l, \id b)$ a fundamental discriminant with $l \in F^\times$ such that
$\cond l$ is prime to $2 \id{d} \id N$, as in Sect.~\ref{sect:half_int}, and
satisfying also the following conditions.

\begin{enumerate}[label=\textbf{Hl\arabic*.},ref={Hl\arabic*}]
	\item The set $\Sigma\ssl = \mathbf{a} \cup \left\{\id p \mid \id N \;:\; 
	  \chic[*]{l}(\id p)^{v\p(\id N)} \varepsilon_g(\id p) = -1 \right\}$ has
	  even cardinality.
    \label{hyp:paridad}
	\item $\sgn(l) = (-1)^{\mathbf{k}}$. \label{hyp:peso_signo}
\end{enumerate}
Such $(l, \id b)$ exists unless $\id N$ is a square and $[F:\Q]$ is odd.
By Hypothesis \ref{hyp:JL} we have that $v\p(\id N)$ is odd for every $\id p \in
\Sigma\ssl$.
Furthermore, assuming that it exists, we fix $u \in \OOx \cap F^l$.
Note that $\sgn(u) = (-1)^\mbk$.

Let $B=B\ssl$ be the quaternion algebra over $F$ ramified exactly at
$\Sigma\ssl$.
Let $\pi$ be the irreducible automorphic representation of
$\GL_2$ corresponding to $g$.
For every prime $\id p$ where $B$ is ramified
$v\p(\id N)$ is odd,
hence the local component of $\pi$ at $\id p$ is square integrable.
It follows that there is an irreducible automorphic representation
$\pi_B$ of $\hatx B$
which corresponds to $\pi$ under the Jacquet-Langlands map.

Let $\mathscr{E}$ denote the set of functions
$\varepsilon:\{\id p\in\mbf \,:\, \id 
p \mid \id N\} \to \{\pm 1\}$ satisfying
\begin{equation}\label{eqn:permitted_eps}
\varepsilon(\id p)^{v\p (\id N)} = \varepsilon_g(\id p) 
\qquad \forall \, \id p \mid \id N\,.
\end{equation}
Note that this set is not empty. This is equivalent to Hypothesis \ref{hyp:JL}.

Fix $\varepsilon \in \mathscr{E}$, and let
$R = R\ssl_{\varepsilon}\subseteq B\ssl$ 
be an order with discriminant $\id N$ and Eichler invariant $e(R\p) = 
\chic[*]{l}(\id p)\,\varepsilon(\id p)$ for every $\id p \mid \id N$. 
Since by $\eqref{eqn:permitted_eps}$ we have that $e(R\p) =
-1$ for $\id p \in \Sigma\ssl$, such order exists, and belongs to the class of
orders considered in Sect.~\ref{sect:orders}.
Furthermore, the lattice $L = R / \OO$ has level $\id N$,
as in Sects.~\ref{sect:half_int} and \ref{sect:special_pts}.
In \cite[Proposition 8.6]{gross-local} it is shown that
$\hatx R$ fixes a unique line in the representation space of $\pi_B$.
This line gives an explicit quaternionic modular form
$\varphi_\varepsilon = \varphi\ssl_{\varepsilon} \in
\mathcal{S}_\mathbf{k}(R,\bbu)$, which is well defined up to a constant.

We let $f_\varepsilon\sslu = \tlu(\varphi_{\varepsilon}) \in \Shk[+]$.
Note that its cuspidality, in the case when Proposition~\ref{prop:cuspidality}
does not apply, was proved in \cite[Proposition 2.5]{waldspu}.

\begin{rmk}\label{rmk:shimura_corresp}
	Let $g' \in \mathcal S_{\mathbf 2 + 2\mbk}(2 \id N)$ be 
	the image of $f\sslu_\varepsilon$ under the Shimura correspondence 
	(as defined in \cite[Theorem 6.1]{shim-hh}).
	We expect that $g'$ has level $\id N$, which by
	Proposition~\ref{prop:hecke_linear} would imply that it is in fact a
	multiple of $g$.
	Furthermore we expect for $u$ and $\varepsilon$ fixed that, up to
	multiplication by a constant, $f\sslu_\varepsilon$ does not depend on $l$.
	When $\id N$ is odd and square-free both facts follow from \cite[Theorem
	9.4]{hiraga-ikeda} and \cite[Theorem 10.1]{su}.
\end{rmk}

\begin{lemma}\label{lem:sect5}
Let $\delta\ssl$ be the character of $\Bil(R)$ given by \eqref{eqn:deltal}.
Then $\varphi_{\varepsilon} \in
\mathcal{S}_\mathbf{k}(R,\mathbbm{1})^{\delta\ssl}$.
\end{lemma}

\begin{proof}
Let $\id p$ be a prime dividing $\id N$, and let $w\p \in N(R\p)$
be the generator for $\Bil(R\p)$ given in \cite[Proposition 4.3]{waldspu}.
Since $\norm(w\p) = - \pi\p^{v\p(\id N)}$
and $\id p \nmid \cond l$, we have that $\delta\ssl(w\p) = \chic[*]{l}(\id
p)^{v\p(\id N)}$.

Furthermore, since $w\p$ has order two and normalizes $\hatx R$, it acts on
$\varphi_\varepsilon$ by multiplication by $\kappa\p \in \{\pm 1\}$.
When $B$ is split at $\id p$ we have $\kappa\p = \varepsilon_g(\id p)$,
and $\kappa\p = -\varepsilon_g(\id p)$ when $B$ is ramified at $\id p$
(for instance, see \cite[Theorem 2.2.1]{roberts-thesis}).

Since by Proposition~\ref{prop:bil_charact} we have that $\{w\p\,:\,\id p\mid\id
N\}$ generates $\Bil(R)$, using that $B$ is ramified at $\id p$ if and only if
$\id p \in \Sigma\ssl$, the result follows.
\end{proof}

Given $D \in F^{-l}$, let $\Delta = l D$ and let $K = F(\sqrt{\Delta})$.
Assume that $\cond D$ is prime to $\cond l \id N$.
We say that $D$ is of \emph{type $\varepsilon$} if $\chic[*]{D}(\id p) =
\varepsilon(\id p)$ for all $\id p\mid\id N$.
Note that for $D$ of type $\varepsilon$
we have
$\chic[*]{\Delta}(\id p)^{v\p (\id N)}
= \chic[*]{l}(\id p)^{v\p (\id N)}\varepsilon_g(\id p)$,
hence the set $\Sigma_{\Delta}$ given in \eqref{eq:SigmaD}
is precisely $\Sigma\ssl$, the ramification of $B\ssl$.
Moreover there exists an embedding $\OO[K]\hookrightarrow R$,
as required in Sect.~\ref{sect:height_geom}.
Furthermore, Hypothesis \ref{hyp:paridad} implies that for such $D$ the sign of 
the functional equation for $L_{l,D}(s,g)$ equals $1$. 

Let $c_g$ be the positive real number given by
\[
 c_g = {\ll{g,g}}\,
 \:
 \frac{\dF^{\;1/2}}{\hF}
 \:
 \frac
 {c(\mathbf{k})\:C(\id N)}
 {2^{2\omegaN}}\,,
\]
where $c(\mathbf{k})$ and $C(\id N)$ are as in Theorem~\ref{thm:xue}.
As required by that theorem, assume for the rest of this section that $\id N
\subsetneq \OO$.

\begin{thm}\label{thm:main_thm}
For every $D \in F^{-l}$ of type $\varepsilon$ such that $\cond D$ is
prime to $2 \cond l \id N$ we have
\begin{equation}\label{eqn:main_thm}
 L_{l,D}(1/2,g) =
 c_g\,
 \frac{c_{\Delta}}{(-\Delta)^\mathbf{k+1/2}}\,
 \frac{
 \abs{ \lambda(-uD,\id a;f_\varepsilon\sslu) }^2}
 {\ll{\varphi_\varepsilon,\varphi_\varepsilon}}
 \,,
\end{equation}
where $\id a\in\IF$ is the unique ideal such that $(D,\id a)$ is a fundamental
discriminant and $c_{\Delta}$ is the positive rational number given by
$c_{\Delta}=\mK^{\;2}\,\idn{ab}$.
\end{thm}

\begin{rmk}
	The sign of the functional equation for $L(s,g\otimes\chi\ssl)$ is equal to
	$\chic[*]{l}(\id N) \sgn(l) (-1)^\mathbf{1+k} \prod_{\id p \mid \id N}
	\varepsilon\subp(g)$.
	Under Hypotheses~\ref{hyp:JL} and \ref{hyp:paridad},
	we see that it  equals $\sgn(l) (-1)^\mbk$. Hence if we do not assume
	Hypothesis~\ref{hyp:peso_signo} both sides of \eqref{eqn:main_thm} vanish
	trivially for every $D$ satisfying the hypotheses above.
\end{rmk}

\begin{proof}

Let $\Tg$ be the polynomial in the Hecke operators prime to $\id N$ giving the 
$g$-isotypical projection.
Let $\psilD$ and $\etalD$ be as in 
Corollary~\ref{coro:psi_vs_eta}.
Since $\Tg\psilD$ is the $\varphi_\varepsilon$-isotypical projection of 
$\psilD$ we have that
$\Tg\psilD = 
\frac{\ll{\psilD,\varphi_\varepsilon}}{\ll{\varphi_\varepsilon,
\varphi_\varepsilon }}
\, \varphi_\varepsilon$.
Combining this with 
Proposition~\ref{prop:coef_serie_theta}, Corollary~\ref{coro:psi_vs_eta} and 
Lemma~\ref{lem:sect5}, we get that
\[
	\ll{\Tg\psilD, \psilD} = 
    \frac{\abs{ \ll{\psilD,\varphi_\varepsilon} }^2}
         {\ll{\varphi_\varepsilon,\varphi_\varepsilon}}
    =
    \frac{\abs{ \ll{\etalD,\varphi_\varepsilon}}^2}
         {2^{2\omegaN} \,
          \ll{\varphi_\varepsilon,\varphi_\varepsilon}}
    =
    \frac{\idn a^2}{2^{2\omegaN}} \,
    \frac{\abs{\lambda(-uD,\id a; f_\varepsilon\sslu)}^2}
         {\ll{\varphi_\varepsilon,\varphi_\varepsilon}}\,.
\]
Then \eqref{eqn:main_thm} follows from Theorem~\ref{thm:xue}.
\end{proof}

\begin{coro}\label{coro:family}
 
Assume that $L(1/2,g\otimes \chic{l}) \neq 0$. Then $f_\varepsilon\sslu \neq 0$.
Moreover, the set $\{f_\varepsilon\sslu \,:\, \varepsilon \in \mathscr{E}\ssl\}$
is linearly independent.

\end{coro}

\begin{proof}
 
Let $D_{\varepsilon} \in F^{-l}$ be of type $\varepsilon$.
Since both signs of the functional equations for $L_{l,D_{\varepsilon}}(s,g)$ 
and $L(s, g\otimes \chic{l})$ are equal to $1$, the same must hold for the 
sign for $L(s, g\otimes \chic{D_\varepsilon})$.
Hence by \cite[Th\'{e}or\`{e}me 4]{waldspu-corresp}  there exists 
$\xi \in F^{l}$
such that if we let $D'_\varepsilon = \xi D_\varepsilon$ 
then $D'_\varepsilon$ is of type $\varepsilon$, the conductor
$\cond{D'_\varepsilon}$ is prime to $2\cond{l}\id N$ and
$L(1/2, (g\otimes \chic{D_\varepsilon}) \otimes \chic{\xi})
= L(1/2, g\otimes \chic{D'_\varepsilon})\neq 0$.
Then by \eqref{eqn:main_thm} we have that 
$\lambda(-uD'_\varepsilon,\id a_\varepsilon;f_\varepsilon\ssl) \neq 0$, where 
$(D'_\varepsilon, \id a_\varepsilon)$ is the discriminant satisfying 
$D'_\varepsilon \id a_\varepsilon^2 = \cond{D'_\varepsilon}$.
This proves the first assertion. 
The second assertion follows from the fact that if $\varepsilon' \neq 
\varepsilon$ then $\lambda(-uD_\varepsilon,\id
a_\varepsilon;f_{\varepsilon'}\sslu)=0$.
\end{proof}

We say that $D\in F^{-l}$ is \emph{permitted}
if $\cond{D}$ is prime to $2\id N$ and 
$\chic[*]{D}(\id p)=\varepsilon_g(\id p)$
for all $\id p\mid\id N$ such that
$v\p(\id N)$ is odd.
By Hypothesis~\ref{hyp:JL},
every permitted $D$ is of type $\varepsilon$
for some $\varepsilon\in\mathscr{E}$.

\begin{coro} \label{coro:main}
There exists $f\sslu \in \Mhk$ such that
 for every permitted $D$ with $\cond D$ prime to $\cond{l}$ we have
 \[
  L_{l,D}(1/2,g)
  = \frac{c_\Delta}{(-\Delta)^\mathbf{k+1/2}}
  \,\abs{ \lambda(-uD,\id a;f\sslu) }^2\,,
 \]
 where $\id a\in\IF$ is the unique
 ideal such that $(D,\id a)$ is a fundamental discriminant.
 Moreover, if $L(1/2,g \otimes \chic{l}) \neq 0$, 
 then $f \sslu \neq 0$.
\end{coro}

\begin{proof}
 This follows from Theorem~\ref{thm:main_thm} and Corollary~\ref{coro:family}, 
 letting
 \[
    f\sslu = c_g^{1/2}\displaystyle\sum_{\varepsilon \in \mathscr{E}} 
\frac{f\sslu_\varepsilon}{\ll{\varphi_\varepsilon,\varphi_\varepsilon}^{1/2}}
    \,.
    \qedhere
 \]
\end{proof}

Finally, we show that under certain hypotheses on the level and weight our
construction gives preimages under the Shimura correspondence.

\begin{coro}\label{coro:shimura_corresp}

Assume that $\id N$ is odd and square-free, and that there exists $u \in \OOx$
such that $\sgn(u) = (-1)^\mbk$. 
Then there exists $l \in F^\times$ such that $f\sslu_\varepsilon$ maps to $g$
under the Shimura correspondence for every $\varepsilon \in \mathcal E$. 

\end{coro}

\begin{proof}

	Choose $l \in F^\times$ as in the beginning of this section (which is
	possible since $\id N$ is not a square) such that $L(1/2,g\otimes \chi\ssl)
	\neq 0$ (which is possible by \cite[Th\'{e}or\`{e}me 4]{waldspu-corresp}).
	Then since $\sgn(l) = (-1)^\mbk$ we have that $u \in F^l$, so that we are in
	the setting of this section. Hence the result follows by Remark
	\ref{rmk:shimura_corresp} and Corollary \ref{coro:family}.
\end{proof}

\section*{Epilogue}

Given $l \in F^\times$, there might not exist $u \in \OOx \cap F\ssl$, which is
required for defining the holomorphic theta series \eqref{eqn:theta}.
Nevertheless, the general results of \cite[Sect.~11]{shim-hh} allow to
consider the theta series
\begin{equation*}
	\vartheta\ssl(z) = \vartheta\ssl(z;w,P)
	= \sum_{y\in \id b^{-1}L} w(y; \id b) \, P(y) 
	\, e_F\ssl(-\disc(y) /l,z/2)\,,
\end{equation*}
where the exponential function is defined by
\[
	e_F\ssl  : F\times \HH^\mathbf{a} \to \C, \quad (\xi, z) \mapsto 
	\exp{\left(2\pi i \Big(
				\sum_{l_\tau > 0} \xi_\tau\, z_\tau +
				\sum_{l_\tau < 0} \xi_\tau\, \overline{z_\tau}
	\Big)\right)}.
\]
This theta series is holomorphic in the variables $\tau \in \mba$ such that
$l_\tau > 0$, and skew-holomorphic in the remaining variables. 
Following the lines of Proposition~\ref{prop:modularity} we can prove that they
are modular, with respect to certain skew-holomorphic automorphy factor
depending on the signs of $l$.

Furthermore, since the results on special points of
Sect.~\ref{sect:special_pts} and the results of
Sects.~\ref{sect:height_geom} and \ref{sect:orders} do not depend on the
auxiliary parameter $u$, we can obtain as in Corollary~\ref{thm:main_thm} a
formula
\begin{equation*}\label{eqn:main_thm_skew}
	 L_{l,D}(1/2,g)=
 \frac{c_{\Delta}}{(-\Delta)^\mathbf{k+1/2}}\,
 \frac{\abs{ \lambda(-D,\id a;f_\varepsilon\ssl) }^2}
      {\ll{\varphi_\varepsilon,\varphi_\varepsilon}}.
\end{equation*}
The only difference is that, since we do not have a result on Fourier coefficients
analogous to \eqref{eqn:cfourier}, we \emph{define} for $\vartheta\ssl$ as above
\[
	\lambda(-D,\id a;\vartheta\ssl) = 
	\frac1{\idn a} \sum_{y \in \mathcal{A}_{\Delta,\id c}(L)}
	w(y;\id c) \, P(y). 
\]

So even though, as far as we know, there is yet no theory developed about the
spaces of these skew-holomorphic modular forms, we can still use them to compute
central values.

\end{document}